\documentclass[12pt]{article}
 \usepackage{amsmath}
 \usepackage{srcltx}
\usepackage{amsfonts}
\usepackage{amssymb}
\usepackage{euscript}
\usepackage{epsfig}
\usepackage{amsthm}
\newtheorem{theorem}{Theorem}[section]
\newtheorem{lemma}[theorem]{Lemma}
\newtheorem{proposition}[theorem]{Proposition}
\newtheorem{corollary}[theorem]{Corollary}
\theoremstyle{remark}

\newtheorem*{remak}{Remark}
\begin{document}

\title{
\large{\textbf{On a class of semihereditary crossed-product orders} } }

\author{\bf \sf John S. Kauta}

\date{}

 \maketitle

\begin{abstract}
Let $F$ be a field, let $V$ be a valuation ring of $F$ of arbitrary
Krull dimension (rank), let $K$ be a finite Galois extension of $F$
with group $G$, and let $S$ be the integral closure of $V$ in $K$.
Let $f:G\times G\mapsto K\setminus \{0\}$ be a normalized
two-cocycle such that $f(G\times G)\subseteq S\setminus \{0\}$, but
we do not require that $f$ should take values in the group of
multiplicative units of $S$. One can construct a crossed-product
$V$-algebra $A_f=\sum_{\sigma\in G}Sx_{\sigma}$ in a natural way,
which is a $V$-order in the crossed-product $F$-algebra $(K/F,G,f)$.
If $V$ is unramified and defectless in $K$, we show that $A_f$ is
semihereditary if and only if for all
$\sigma,\tau\in G$ and every maximal ideal $M$ of $S$, $f(\sigma,\tau)\not\in M^2$. If in
addition $J(V)$ is not a principal ideal of $V$, then $A_f$ is
semihereditary if and only if it is an Azumaya algebra over $V$.


\end{abstract}

\section{Introduction}
\label{sec:intro}
In this paper we study certain orders over valuation rings in central simple
algebras. If $R$ is a ring, then $J(R)$ will
denote its Jacobson radical, $U(R)$ its group of multiplicative
units, and $R^{\#}$ the subset of all the non-zero elements. The
residue ring $R/J(R)$ will be denoted by $\overline{R}$. Given the
ring $R$, it is called \textit{primary} if $J(R)$ is a maximal ideal
of $R$. It is called \textit{hereditary} if one-sided ideals are
projective $R$-modules. It is called \textit{semihereditary}
(respectively \textit{B\'{e}zout}) if finitely generated
one-sided ideals are projective $R$-modules (respectively are principal).
Let $V$ be a valuation ring of a field $F$. If $Q$ is
a finite-dimensional central simple $F$-algebra, then a subring $R$
of $Q$ is called an order in $Q$ if $RF=Q$. If in addition
$V\subseteq R$ and $R$ is integral over $V$, then $R$ is called a
$V$-order. If a $V$-order $R$ is maximal among the $V$-orders of $Q$
with respect to inclusion, then $R$ is called a maximal $V$-order
(or just a maximal order if the context is clear). A $V$-order $R$
of $Q$ is called an \textit{extremal} $V$-order (or simply \textit{extremal}
when the context is clear) if for every $V$-order $B$ in $Q$ with
$B\supseteq R$ and $J(B)\supseteq J(R)$, we have $B=R$. If $R$ is an
order in $Q$, then it is called a \textit{Dubrovin valuation ring} of $Q$
(or a \textit{valuation ring} of $Q$ in short) if
it is semihereditary and primary (see \cite{D1,D2}).

In this paper, $V$ will denote a commutative valuation ring of \textit{arbitrary}
Krull dimension (rank). Let $F$ be its field of quotients, let $K/F$
be a finite Galois extension with group $G$, and let $S$ be the
integral closure of $V$ in $K$. If $f\in Z^2(G,U(K))$ is a
normalized two-cocycle such that $f(G\times G)\subseteq S^{\#}$,
then one can construct a ``crossed-product'' $V$-algebra
$$A_f=\sum_{\sigma\in G} Sx_{\sigma},$$ with the usual rules of multiplication
($x_{\sigma}s=\sigma(s)x_{\sigma}$ for all $s\in S,\sigma\in G$ and
$x_{\sigma}x_{\tau}=f(\sigma,\tau)x_{\sigma\tau})$. Then $A_f$ is
associative, with identity $1=x_1$, and center $V=Vx_1$. Further,
$A_f$ is a $V$-order in the crossed-product $F$-algebra
$\Sigma_f=\sum_{\sigma\in G} Kx_{\sigma}=(K/F,G,f)$. Following
\cite{H}, we let $H=\{\sigma\in G\mid f(\sigma,\sigma^{-1})\in
U(S)\}$. Then $H$ is a subgroup of $G$.

In this paper, we will \textit{always} assume that $V$ is unramified
and defectless in $K$ (for the definitions of these terms, see \cite{E}).
By \cite[Theorem 18.6]{E}, $S$ is a finitely generated
$V$-module, hence $A_f$ is always finitely generated over $V$. If
$V_1$ is a valuation ring of $K$ lying over $V$ then $\{\sigma\in
G\mid \sigma(x)-x\in J(V_1)\;\forall\; x\in V_1\}$ is called the
\textit{inertial group} of $V_1$ over $F$. By \cite[Lemma 1]{K3},
the condition that $V$ is unramified and defectless in $K$ is
equivalent to saying that the inertial group of $V_1$ over $F$ is
trivial, since $K/F$ is a finite Galois extension.

These orders were first studied in \cite{H}, and later in \cite{HM}
and \cite{K4}. In \cite{H} and \cite{K4}, only the case when $V$ is
a discrete valuation ring (DVR) was considered. In \cite{K4},
hereditary properties of crossed-product orders were examined. In
\cite{H} and \cite{HM}, valuation ring properties of the
crossed-product orders were explored, and the latter considered the
cases when either $V$ had arbitrary Krull dimension but was
indecomposed in $K$, or $V$ was a discrete finite rank valuation
ring, that is, its value group is
$\mathbb{Z}\oplus\cdots\oplus\mathbb{Z}$. When $V$ is a DVR, then
any $V$-order in $\Sigma_f$ containing $S$ is a crossed-product
order of the form $A_g$ for some two-cocycle $g:G\times G\mapsto
S^{\#}$, with $g$ cohomologous to $f$ over $K$, by \cite[Proposition
1.3]{H}, but this need not be the case in general. While \cite{HM}
considered any $V$-order in $\Sigma_f$ containing $S$, some of which
were not of the type described above and so in that sense its scope
was wider than ours, in this paper we shall only be concerned with
crossed-product orders $A_g$ where $g$ is either $f$ (almost
always), or is cohomologous to $f$ over $K$, that is,  if there are elements
$\{c_{\sigma}\mid \sigma\in G\}\subseteq K^{\#}$ such that
$g(\sigma,\tau)=c_{\sigma}\sigma{(c_{\tau})}c_{\sigma\tau}^{-1}f(\sigma,\tau)$
for all $\sigma,\tau \in G$, a fact denoted by
$g\!\sim_K\! f$.

The purpose of this paper is to generalize the results of \cite{K4}
to the case when $V$ is not necessarily a DVR. The main results of this
paper are as follows: $A_f$ is semihereditary if and only if
for all
$\sigma,\tau\in G$ and every maximal ideal $M$ of $S$, $f(\sigma,\tau)\not\in M^2$;
if $J(V)$ is not a principal ideal of $V$,
then $A_f$ is semihereditary if and only if it is an Azumaya algebra
over $V$. As in \cite{K4}, the utility of these criteria lie in
their simplicity.

Although in our case the valuation ring $V$ need not be a DVR, some of
the steps in the proofs in \cite{H} and \cite{K4} remain valid,
\textit{mutatis mutandis}, owing to the theory developed in
\cite{K1,K2}. We shall take full advantage of this whenever the
opportunity arises. Aside from the difficulties inherit when dealing
with $V$-orders that are not necessarily noetherian, the hurdles
encountered in this theory arise mainly due to the fact that
the two-cocycle $f$ is not assumed to take on values in $U(S)$.

\section{Preliminaries}
\label{sec:pre}

In this section, we gather together various results that will help
us prove the main results of this paper, which are in the next
section. For the convenience of the reader, we have included
complete proofs whenever it warrants, although the arguments are sometimes routine.

 The following lemma is essentially embedded in the
proof of \cite[Proposition 1.8]{K1}, and the remark that follows it.

\begin{lemma} \label{lem:fgextremal-lemma}
Let $A$ be a finitely generated extremal $V$-order in a
finite-dimensional central simple $F$-algebra $Q$.
\begin{enumerate}
    \item If $B$ is a $V$-order of $Q$ containing $A$, then $B$ is
    also a finitely generated extremal order. If in addition $B$ is
    a maximal $V$-order, then it is a valuation ring of $Q$.
    \item If $W$ is an overring of $V$ in $F$ with $V\subsetneqq W$,
     then $WA$ is a valuation ring of $Q$ with center $W$.
\end{enumerate}
\end{lemma}

\begin{proof} Let $B$ be a $V$-order containing $A$. By \cite[Proposition 1.8]{K1},
$A$ is semihereditary, hence $B$ is semihereditary by \cite[Lemma
4.10]{M}, and therefore $B$ is extremal by \cite[Theorem 1.5]{K1}.
Since $[B/J(B):V/J(V)]\leq [\Sigma_f:F]<\infty$, there exists
$a_1,a_2,\ldots,a_m\in B$ such that $B=a_1V+a_2V+\cdots +a_mV+
J(B)$. But by \cite[Proposition 1.4]{K1}, $J(B)\subseteq J(A)$,
since $A$ is extremal. Therefore $B=a_1V+a_2V+\cdots +a_mV+ A$,
a finitely generated $V$-order. If, in addition $B$ is a maximal
$V$-order, then by the remark after \cite[Proposition 1.8]{K1}, $B$
is a valuation ring of $Q$.

Now let $W$ be a proper overring of $V$ in $F$. Let $C$ be a maximal
$V$-order containing $A$. Then $C$ is a valuation ring of $Q$, as seen above,
hence $WC$ is a valuation ring of $Q$ with center $W$.
Since $A$ is an extremal $V$-order, we have $J(C)\subseteq J(A)$,
thus $WC=WJ(V)C\subseteq WJ(C)\subseteq WA\subseteq WC$, so that
$WA=WC$. Thus $WA$ is always a valuation ring of $Q$.
\end{proof}

Since $A_f$ is finitely generated over $V$, we immediately have the
following lemma, because of \cite[Proposition 1.8]{K1}, the remark
that follows it, and the fact that B\'{e}zout $V$-orders are maximal
orders by \cite[Theorem 3.4]{M}.

\begin{lemma} \label{lem:extremal-maximal-lemma}
Given the crossed-product order $A_f$,
\begin{enumerate}
\item it is an extremal order if and only if it is
semihereditary.
\item it is a maximal order if and only if it is a
valuation ring, if and only if it is B\'{e}zout.
\end{enumerate}
\end{lemma}

\begin{lemma} \label{lem:overring-lemma}
Let $W$ be a valuation ring of $F$ such that $V\subsetneqq W$, and
let $R=WS$.
\begin{enumerate}
    \item Then $R$ is the integral closure of $W$ in $K$,
    and $W$ is also unramified and defectless in $K$.
    \item Let $t\in S$ satisfy $t\not\in M^2$ for every maximal ideal
    $M$ of $S$. Then $t\in U(R)$. If in addition $J(V)$ is a non-principal
    ideal of $V$, then $t\in U(S)$.
\end{enumerate}
\end{lemma}

\begin{proof} The ring $R$ is obviously integral over $W$. Since it
contains $S$, it is also integrally closed in $K$, hence it is the
integral closure of $W$ in $K$.

Now let $V_1\subseteq W_1$ be valuation rings of $K$ lying over $V$
and $W$ respectively. Then $J(W_1)\subseteq J(V_1)$, hence the
inertial group of $W_1$ over $F$, $\{\sigma\in G\mid \sigma(x)-x\in
J(W_1)\;\forall\; x\in W_1\}$, is contained in the inertial group of
$V_1$ over $F$, $\{\sigma\in G\mid \sigma(x)-x\in J(V_1)\;\forall\;
x\in V_1\}$. Since $V$ is unramified and defectless in $K$, the
latter group is trivial, forcing $W$ to be unramified and defectless
in $K$.

Let $W_1$ be a valuation ring of $K$ lying over $W$, and let $V_1$
be a valuation ring of $K$ lying over $V$ such that $V_1\subseteq
W_1$, as in the preceding paragraph. Let $M=J(V_1)\cap S$, a generic maximal
ideal of $S$. We claim that $M^2=J(V_1)^2\cap S$. To see this, note
that $M^2=(J(V_1)\cap S)(J(V_1)\cap S)\subseteq J(V_1)^2\cap S$, and
$M^2V_1=(J(V_1)\cap S)(J(V_1)\cap S)V_1=J(V_1)^2=(J(V_1)^2\cap
S)V_1$. If $V'$ is an extension of $V$ to $K$ different from $V_1$,
then $M^2V'=V'=(J(V_1)^2\cap S)V'$. Thus $M^2=J(V_1)^2\cap S$ as
desired. If $t\in S$ satisfies $t\not\in M^2$, then $t\not\in
J(V_1)^2$. Since $J(W_1)\subsetneqq J(V_1)^2$, we have $t\in
U(W_1)$. Since $W_1$ was an arbitrary extension of $W$ in $K$, we
conclude that $t\in U(R)$. If $J(V)$ is a non-principal ideal of
$V$, then $J(V_1)^2=J(V_1)$, hence $t\in U(V_1)$ for every such
extension $V_1$ of $V$ to $K$, and we conclude that $t\in U(S)$.
\end{proof}

Part 4 of the following lemma was originally proved in \cite{H} when
$V$ is a DVR. The same arguments work when $V$ is an arbitrary
valuation ring.

\begin{lemma} \label{lem:radical-lemma}
Given a $\sigma\in G$, let $I_{\sigma}=\cap M$, where
the intersection is taken over those maximal ideals $M$ of $S$ for
which $f(\sigma,\sigma^{-1})\not\in M$. Then
\begin{enumerate}
    \item $I_{\sigma}=\{x\in S\mid xf(\sigma,\sigma^{-1})\in
    J(V)S\}$.
    \item $I_{\sigma}^{\sigma^{-1}}=I_{\sigma^{-1}}$.
    \item If $f(\sigma,\sigma^{-1})\not\in M^2$ for every maximal
    ideal $M$ of $S$, then $I_{\sigma}f(\sigma,\sigma^{-1})=J(V)S$.
    \item $J(A_f)=\sum_{\sigma\in G} I_{\sigma}x_{\sigma}$.
\end{enumerate}
 \end{lemma}

 \begin{proof} Let $x\in S$. Clearly, if $x\in I_{\sigma}$ then $xf(\sigma,\sigma^{-1})\in
 J(V)S$. On the other hand, if $x\not\in I_{\sigma}$ then there
 exists a maximal ideal $M$ of $S$ such that
 $x,f(\sigma,\sigma^{-1})\not\in M$, hence
 $xf(\sigma,\sigma^{-1})\not\in M$, and thus $xf(\sigma,\sigma^{-1})\not\in J(V)S$.

 The second statement is proved in the same manner as
 \cite[Sublemma]{K4}. To see that the third statement holds,
we note that $I_{\sigma}f(\sigma,\sigma^{-1})\subseteq J(V)S$. We
claim that $I_{\sigma}f(\sigma,\sigma^{-1})= J(V)S$. To see this,
let $M$ be a maximal ideal of $S$. If $f(\sigma,\sigma^{-1})\not\in
M$, then $(I_{\sigma}f(\sigma,\sigma^{-1}))S_M=J(S_M)=(J(V)S)S_M$.
On the other hand, if $f(\sigma,\sigma^{-1})\in M$ then, since
$f(\sigma,\sigma^{-1})\not\in M^2$, we have $J(S_M)^2 \subsetneqq
I_{\sigma}f(\sigma,\sigma^{-1})S_M\subseteq J(S_M)$, hence
$I_{\sigma}f(\sigma,\sigma^{-1})S_M=J(S_M)=(J(V)S)S_M$, and thus
$I_{\sigma}f(\sigma,\sigma^{-1})= J(V)S$. By \cite[Lemma 1.3]{HM},
$J(A_f)=\sum_{\sigma\in G}(J(A_f)\cap Sx_{\sigma})$. Therefore the
fourth statement can
be verified in exactly the same manner as \cite[Proposition
3.1(b)]{H}, because of the observations made above.
\end{proof}

The following lemma is a generalization of  \cite[Proposition
1.3]{H}.

\begin{lemma} \label{lem:cross-prod-overring-lemma}
Let $B\subseteq\Sigma_f$ be a $V$-order. There is a normalized
cocycle $g:G\times G\mapsto S^{\#}$, $g\sim_K f$, such that $B=A_g$
(viewed as a subalgebra of $\Sigma_f$ in a natural way) if and only
if $B\supseteq S$ and $B$ is finitely generated over $V$. When this
occurs, $B=\sum_{\sigma\in G}Sk_{\sigma}x_{\sigma}$ for some
$k_{\sigma}\in K^{\#}$.
\end{lemma}

\begin{proof} Suppose $B\supseteq S$.
By \cite[Lemma 1.3]{HM}, $B=\sum_{\sigma\in G}B_{\sigma}x_{\sigma}$,
where each $B_{\sigma}$ is a non-zero $S$-submodule of $K$. If in
addition $B$ is finitely generated over $V$, then each $B_{\sigma}$
is finitely generated over $V$: if $B=\sum_{i=1}^nVy_i$ then, if we
write $y_i=\sum_{\tau\in G}k_{\tau}^{(i)}x_{\tau}$ with
$k_{\tau}^{(i)}\in K$, we see that $B_{\sigma}$ is generated by
$\{k_{\sigma}^{(i)}\}_{i=1}^{n}$ over $V$. Since $S$ is a
commutative B\'{e}zout domain with $K$ as its field of quotients,
$B_{\sigma}=Sk_{\sigma}$ for some $k_{\sigma}\in K^{\#}$. Thus we
get $B=\sum_{\sigma\in G}Sk_{\sigma}x_{\sigma}$. Since $B$ is integral
over $V$, $B_1=S$ and so we can choose $k_1=1$.
Define $g:G\times
G\mapsto S^{\#}$ by
$g(\sigma,\tau)k_{\sigma\tau}x_{\sigma\tau}=(k_{\sigma}x_{\sigma})(k_{\tau}x_{\tau})$,
as in \cite[Proposition 1.3]{H}. Since $k_1=1$, $g$ is also a normalized
two-cocycle.
The converse is obvious.
\end{proof}

\begin{lemma} \label{lem:Azumaya-lemma}
Suppose $S$ is a valuation ring of $K$. Then the following are
equivalent:
\begin{enumerate}
    \item $J(V)A_f$ is a maximal ideal of $A_f$.
    \item $H=G$.
    \item $A_f$ is Azumaya over $V$.
\end{enumerate}
\end{lemma}

\begin{proof} Suppose $J(V)A_f$ is a maximal ideal of $A_f$. Note that
$A_f/J(V)A_f=\sum_{\sigma\in G}\overline{S}\tilde{x}_{\sigma}.$ By
\cite[Theorem 10.1(c)]{HLS}, $J=\sum_{\sigma\not\in
H}\overline{S}\tilde{x}_{\sigma}$ is an ideal of $A_f/J(V)A_f$.
Since $A_f/J(V)A_f$ is simple, $J=0$, hence $H=G$. \end{proof}

We set up additional notation, following \cite{H} and \cite{K4}. Let
$L$ be an intermediate field of $F$ and $K$, let $G_L$ be the Galois
group of $K$ over $L$, let $U$ be a valuation ring of $L$ lying over
$V$, and let $T$ be the integral closure of $U$ in $K$. Then one can
obtain a two-cocycle $f_{L,U}:G_L\times G_L\mapsto T^{\#}$ from $f$
by restricting $f$ to $G_L\times G_L$, and embedding $S^{\#}$ in
$T^{\#}$. As before, $A_{f_{L,U}}=\sum_{\sigma\in G_L}Tx_{\sigma}$
is a $U$-order in $\Sigma_{f_{L,U}}=\sum_{\sigma\in
G_L}Kx_{\sigma}=(K/L,G_L,f_{L,U})$, and $U$ is unramified and
defectless in $K$. If $M$ is a maximal ideal of $S$, and $L$ is the
decomposition field of $M$ and $U=L\cap S_M$, then we will denote
$f_{L,U}$ by $f_M$, $A_{f_{L,U}}$ by $A_{f_M}$, $\Sigma_{f_{L,U}}$
by $\Sigma_{f_M}$, $L$ by $K_M$, and the decomposition group $G_L$
by $D_M$, as in \cite{H}. Further, we let $H_M=\{\sigma\in D_M\mid
f_M(\sigma,\sigma^{-1})\in U(S_M)\}$, a subgroup of $D_M$.

Given a maximal ideal $M$ of $S$, let $M=M_1,M_2,\ldots, M_r$ be the
complete list of maximal ideals of $S$, let $U_i=S_{M_i}\cap
K_{M_i}$ with $U=U_1$, and let $(K_i,S_i)$ be a Henselization of
$(K,S_{M_i})$. Let $(F_h,V_h)$ be the unique Henselization of
$(F,V)$ contained in $(K_1,S_1)$. We note that $(F_h,V_h)$ is also a
Henselization of $(K_M,U)$. By \cite[Proposition 11]{HMW}, we have
$S\otimes_VV_h\cong S_1\oplus S_2\oplus\cdots\oplus S_r$.

Part (1) of the following lemma was originally proved in \cite{H} in
the case when $V$ is a DVR. Virtually the same proof holds in the
general case. Part (2)(c) is a generalization of \cite[Corollary 3.11]{H}.

\begin{lemma} \label{lem:primary-lemma}
With the notation as above, we have
\begin{enumerate}
    \item the crossed-product order $A_f$ is primary if and only if
    for every maximal ideal $M$ of $S$ there is a set of right coset
    representatives $g_1,g_2,\ldots, g_r$ of $D_M$ in $G$ (i.e., $G$
    is the disjoint union $\cup_iD_Mg_i)$ such that for all i,
    $f(g_i,g_i^{-1})\not\in M.$
    \item if the crossed-product order $A_f$ is primary, then
    \begin{enumerate}
        \item $A_f\otimes_VV_h\cong M_r(A_{f_{M}}\!\!\otimes_{U}\!\!V_h)$, hence
        \item $A_f/J(A_f)\cong M_r(A_{f_M}/J(A_{f_M}))$, and
        \item $A_f$ is a valuation ring of $\Sigma_f$ if and only if
        $A_{f_M}$ is a valuation ring of $\Sigma_{f_M}$ for some maximal ideal $M$ of $S$.
        When this occurs,
        $A_{f_M}$ is a valuation ring of $\Sigma_{f_M}$ for every maximal ideal $M$ of
        $S$.
        \item $A_f$ is Azumaya over $V$ if and only if
        $H_M=D_M$ for some maximal ideal $M$ of $S$. When this occurs,
        $H_M=D_M$ for every maximal ideal $M$ of $S$.
    \end{enumerate}
\end{enumerate}
\end{lemma}

\begin{proof} The proof of \cite[Theorem 3.2]{H}, appropriately adapted, works here as well to establish part (1). We outline the argument, for the convenience of the reader: For a $\sigma\in G$, let $I_{\sigma}$ be as in Lemma~\ref{lem:radical-lemma},
and, for a maximal ideal $M$ of $S$, set $\hat{M}:=\cap_{N\;
\textrm{max},\; N\not=M}N$. If $I$ is an ideal of $A_f$ then, by
\cite[Lemma 1.3]{HM}, $I=\sum_{\sigma\in G}(I\cap Sx_{\sigma})$, so
$A_f$ is primary if and only if the following condition holds: if
$\sigma\in G$ and $T$ is an ideal of $S$ such that $T\not\subseteq
I_{\sigma}$, then $A_fTx_{\sigma}A_f=A_f$.

If $A_f$ is primary and $M$ is a maximal ideal of $S$, then $A_f=A_f\hat{M}x_1A_f$. Therefore if $G=\cup_{j=1}^rh_jD_M$ is a left coset decomposition, then $$S=\sum_j\hat{M}^{h_j}\left(\sum_{d\in D_M}f(h_jd,d^{-1}h_j^{-1})\right)$$ as in the proof of \cite[Theorem 3.2]{H}, so that, if we fix $i$, $1\leq i\leq r$, and localize at $M^{h_i}$, we get
$$S_{M^{h_i}}=\sum_{j\not=i}J(S_{M^{h_i}})\left(\sum_{d\in D_M}f(h_jd,d^{-1}h_j^{-1})\right)+
S_{M^{h_i}}\left(\sum_{d\in D_M}f(h_id,d^{-1}h_i^{-1})\right),$$ and hence $\sum_{d\in D_M}f(h_id,d^{-1}h_i^{-1})\not\in M^{h_i}$. So there is an element $d_i\in D_M$ such that
$f(h_id_i,d_i^{-1}h_i^{-1})\not\in M^{h_i}$. Let $g_i=d_i^{-1}h_i^{-1}$. Then $g_1,g_2,\ldots,g_r$ have the desired properties.

For the converse, suppose $\sigma\in G$ and $T$ is an ideal of $S$ such that $T\not\subseteq I_{\sigma}$. We need to show that $A_fTx_{\sigma}A_f=A_f$. Since $T\not\subseteq I_{\sigma}$, there is a maximal ideal $M$ of $S$ such that $f(\sigma,\sigma^{-1})\not\in M$ and $T\not\subseteq M$. The argument in \cite[Theorem 3.2]{H} shows that
$A_fTx_{\sigma}A_f\supseteq \sum_{i=1}^rT_i$, where $T_i=T^{g_i^{-1}}f^{g_i^{-1}}(\sigma,\sigma^{-1}g_i)f(g_i^{-1},g_i)$ are ideals of $S$ satisfying the condition $T_i\not\subseteq M^{g_i^{-1}}$. Inasmuch as $g_1^{-1},g_2^{-1},\ldots,g_r^{-1}$ form a complete set of \textit{left} coset representatives of $D_M$ in $G$, the ideal $\sum_{i=1}^rT_i$ is not contained in any maximal ideal of $S$. Therefore $\sum_{i=1}^rT_i=S$, and so $A_fTx_{\sigma}A_f=A_f$.

Using part (1) and the fact that $S\otimes_VV_h\cong
S_1\oplus S_2\oplus\cdots\oplus S_r$, we can construct a full set of
matrix units in $A_f\otimes_VV_h$ and hence verify part
(2)(a), as in the proof of \cite[Theorem 3.12]{H} (see also
the remark after \cite[Theorem 3.12]{H}). Part (2)(b) follows from (2)(a) and \cite[Lemma 3.1]{K1}; part
(2)(c) follows from (2)(a); and (2)(d)
follows from (2)(a) and Lemma~\ref{lem:Azumaya-lemma}.
\end{proof}

\section{The Main Results}
\label{sec:main-results}

We now give the main results of this paper. There are essentially
two parallel theories: one takes effect when $J(V)$ is a principal
ideal of $V$, and the other when it is not. In the former case, the
order $A_f$ displays characteristics akin to the situation when $V$
is a DVR. Our theory, however, yields surprising results in the
latter case. It turns out in this case that the property that $A_f$
is Azumaya over $V$ is equivalent to a much more weaker property:
that it is an extremal $V$-order in $\Sigma_f$.

\begin{proposition} \label{prop:Azumaya-prop}
The order $A_f$ is Azumaya over $V$ if and only if
$H=G$. \end{proposition}

\begin{proof}
Suppose $A_f$ is Azumaya over $V$. Let $M$ be a maximal ideal of
$S$. By Lemma~\ref{lem:primary-lemma}(1), there is a set of right coset representatives
$g_1,g_2,\ldots, g_r$ of $D_M$ in $G$ such that
$f(g_i,g_i^{-1})\not\in M$. If $\sigma\in G$, then $\sigma=hg_i$ for
some $h\in D_M$ and some $i$. Since $A_f$ is Azumaya, $H_M=D_M$ by
Lemma~\ref{lem:primary-lemma}(2)(d), hence we have $f(h^{-1},h)\not\in M$. Because
$$f^{h^{-1}}(hg_i,g_i^{-1}h^{-1})f^{h^{-1}}(h,g_i)f^{g_i}(g_i^{-1},h^{-1})=f(h^{-1},h)f(g_i,g_i^{-1}),$$
we conclude that $f(\sigma,\sigma^{-1})\not\in M$. Since $M$ is
arbitrary, $f(\sigma,\sigma^{-1})\in U(S)$ for every $\sigma\in G$,
so that $H=G$.

The converse is well-known and straightforward to demonstrate.

\end{proof}

It is perhaps instructive to compare the above proposition to
\cite[Theorem 3]{K3}.

Recall that $J(V)$ is a non-principal ideal of $V$ if and only if
$J(V)^2=J(V)$.

\begin{proposition} \label{prop:nonprincipal-characterization}
Suppose $J(V)$ is a non-principal ideal of $V$. Then the following
statements about the crossed-product order $A_f$ are equivalent:
\begin{enumerate}
    \item $A_f$ is an extremal $V$-order in $\Sigma_f$.
    \item $A_f$ is a semihereditary
    $V$-order.
    \item $A_f$ is a maximal $V$-order in $\Sigma_f$.
    \item $A_f$ is a B\'{e}zout $V$-order.
    \item $A_f$ is a valuation ring of
    $\Sigma_f$.
    \item $A_f$ is Azumaya over $V$.
\end{enumerate}
\end{proposition}

\begin{proof} By Lemma~\ref{lem:extremal-maximal-lemma}, it suffices to demonstrate that $\textrm{(1)}
\Longrightarrow \textrm{(5)} \Longrightarrow \textrm{(6)}$. So suppose
$A_f$ is an extremal $V$-order. Let $B$ be a maximal $V$-order
containing $A_f$. By Lemma~\ref{lem:fgextremal-lemma}, $B$ is a valuation ring finitely
generated over $V$. By Lemma~\ref{lem:cross-prod-overring-lemma}, we get that $B=\sum_{\sigma\in
G}Sk_{\sigma}x_{\sigma}$ for some $k_{\sigma}\in K^{\#}$. Since
$A_f$ is extremal, we have $J(B)\subseteq J(A_f)$ by
\cite[Proposition 1.4]{K1}, so $J(V)B\subseteq A_f$. Therefore
$\sum_{\sigma\in G}J(S)k_{\sigma}x_{\sigma}=J(V)B=J(V)^2B\subseteq
J(V)A_f=\sum_{\sigma\in G}J(S)x_{\sigma}$, so that
$J(S)k_{\sigma}\subseteq J(S)$. Hence for each maximal ideal $M$ of
$S$, we have $S_MJ(S)k_{\sigma}\subseteq S_MJ(S)$, that is,
$J(S_M)k_{\sigma}\subseteq J(S_M)$. This shows that $k_{\sigma}\in
S_M$ for all $M$ and so $k_{\sigma}\in S$ for every $\sigma\in G$,
and thus $A_f=B$, a valuation ring.

Now suppose $A_f$ is a valuation ring of $\Sigma_f$. By Lemma
2.7(2), to show that $A_f$ is Azumaya over $V$, we may as well
assume $S$ is a valuation ring of $K$. By \cite[\S 2, Theorem
1]{D2}, $J(A_f)=J(V)A_f$, and so $A_f$ is Azumaya over $V$ by Lemma
2.6.

\end{proof}

\begin{remak} \label{rem:nonprincipal-criterion-remark}
It follows from Lemma~\ref{lem:overring-lemma}(2) and Proposition~\ref{prop:Azumaya-prop} that, if
$J(V)$ is a non-principal ideal of $V$, then the crossed-product
order $A_f$ is extremal if and only if for all
$\tau,\gamma\in G$ and every maximal ideal $M$ of $S$, $f(\tau,\gamma)\not\in M^2$.
\end{remak}

If $W$ is a valuation ring of $F$ such that $V\subsetneqq W$, then
we will denote by $B_f$ the $W$-order $WA_f=\sum_{\sigma\in
G}Rx_{\sigma}$, where $R=WS$ is the integral closure of $W$ in $K$ by
Lemma~\ref{lem:overring-lemma}. Recall that $W$ is also unramified and defectless in $K$.

\begin{proposition} \label{prop:principal-criterion}
Suppose $J(V)$ is a principal ideal of $V$. Then $A_f$ is
semihereditary if and only if for all
$\tau,\gamma\in G$ and every maximal ideal $M$ of $S$, $f(\tau,\gamma)\not\in M^2$.
\end{proposition}

\begin{proof}
The result holds when the Krull dimension of $V$ is one, by
\cite[Corollary]{K4}, since $V$ is a DVR in this case. So let us
assume from now on that the Krull dimension of $V$ is greater than
one.

Let $p=\cap_{n\geq 1}J(V)^n$. Then $p$ is a prime ideal of $V$,
$W=V_p$ is a minimal overring of $V$ in $F$, and $\tilde{V}=V/J(W)$
is a DVR of $\overline{W}$. Set $B_f=WA_f$, as above.

Suppose $A_f$ is semihereditary. We will show that for each $\tau\in G$ and each
maximal ideal $M$ of $S$, $f(\tau,\tau^{-1})\not\in M^2$.

First, assume that $V$ is indecomposed in $K$. By \cite[Proposition
2.6]{HM}, $A_f$ is primary, hence it is a valuation ring of
$\Sigma_f$. Therefore $B_f$ is Azumaya over $W$, by
\cite[Proposition 2.10]{HM}, and $f(G\times G)\subseteq U(R)$, by Proposition~\ref{prop:Azumaya-prop}.
Observe that $R$ is a valuation ring of
$K$ lying over $W$ and $\overline{R}$ is Galois over $\overline{W}$,
with group $G$, and $B_f/J(B_f)=\sum_{\sigma\in
G}\overline{R}\tilde{x}_{\sigma}$ is a crossed-product
$\overline{W}$-algebra. Further, $A_f/J(B_f)$ has center
$\tilde{V}$, a DVR of $\overline{W}$, and is a crossed-product
$\tilde{V}$-order in $B_f/J(B_f)$ of the type under consideration in this
paper, since $\tilde{V}$ is unramified in $\overline{R}$ and
$f(G\times G)\subseteq S\cap U(R)$. Since the
crossed-product $\tilde{V}$-order $A_f/J(B_f)$ is a valuation ring
of $B_f/J(B_f)$ hence hereditary, it follows from \cite[Theorem]{K4}
that for each $\tau\in G$, $f(\tau,\tau^{-1})\not\in J(S)^2$.

Suppose $V$ is not necessarily indecomposed in $K$, but assume $A_f$
is a valuation ring. Fix a maximal ideal $M$ of $S$. By Lemma~\ref{lem:primary-lemma}(1),
there is a set of right coset representatives $g_1,g_2,\ldots, g_r$
of $D_M$ in $G$ such that $f(g_i,g_i^{-1})\not\in M$. If $\tau\in
G$, then $\tau=hg_i$ for some $h\in D_M$ and some $i$. By Lemma~\ref{lem:primary-lemma}(2),
$A_{f_M}$ is a valuation ring of $\Sigma_{f_M}$. Hence, by the
preceding paragraph, $f_M(h^{-1},h)\not\in M^2$, and thus
$f(h^{-1},h)\not\in M^2$. But the following holds:
$$f^{h^{-1}}(hg_i,g_i^{-1}h^{-1})f^{h^{-1}}(h,g_i)f^{g_i}(g_i^{-1},h^{-1})=
f(h^{-1},h)f(g_i,g_i^{-1}).$$ Therefore we must have
$f(\tau,\tau^{-1})\not\in M^2$.

Now suppose that $A_f$ is not necessarily a valuation ring. To show
that for each $\tau\in G$ and each
maximal ideal $M$ of $S$ we have $f(\tau,\tau^{-1})\not\in M^2$,
 one only needs to emulate the corresponding
steps in the proof of \cite[Theorem]{K4}, equipped with the
following four observations: 1) Any maximal $V$-order containing
$A_f$ is a valuation ring, by Lemma~\ref{lem:fgextremal-lemma}, hence $A_f$ is the
intersection of finitely many valuation rings all with center $V$,
since $J(V)$ is a principal ideal of $V$, by \cite[Theorem 2.5]{K2}.
2) If $B$ is one such valuation ring containing $A_f$, then
$B=A_g=\sum_{\tau\in G}Sk_{\tau}x_{\tau}$ for
 some $k_{\tau}\in K^{\#}$, where $g:G\times G\mapsto S^{\#}$ is
 some normalized two-cocycle, by Lemma~\ref{lem:fgextremal-lemma}(1) and Lemma~\ref{lem:cross-prod-overring-lemma}. Fix
 $\sigma\in G$ and a maximal ideal $N$ of $S$. We may choose
 $B$ such that $k_{\sigma}\in U(S_N)$, as in the proof of \cite[Theorem]{K4}.
 3) Both $J(A_f)$ and $J(A_g)$ are as
 in Lemma~\ref{lem:radical-lemma}, that is, $J(A_f)=\sum_{\sigma\in
 G}I_{\sigma}x_{\sigma}$ (respectively $J(B_f)=\sum_{\sigma\in
 G}J_{\sigma}k_{\sigma}x_{\sigma}$) where $I_{\sigma}=\cap M$
 (respectively $J_{\sigma}=\cap M$), as $M$ runs through all
 maximal ideals of $S$ for which $f(\sigma,\sigma^{-1})\not\in M$
 (respectively $g(\sigma,\sigma^{-1})\not\in M$). We have $J(A_g)\subseteq J(A_f)$
 by \cite[Theorem 1.5]{K1}. 4) By Lemma~\ref{lem:radical-lemma},
 $I_{\sigma}^{\sigma^{-1}}=I_{\sigma^{-1}}$,
 $J_{\sigma}^{\sigma^{-1}}=J_{\sigma^{-1}}$, and
 $J_{\sigma^{-1}}g(\sigma^{-1},\sigma)=J(V)S$.

 We conclude, as in the proof of \cite[Theorem]{K4}, that
  \begin{eqnarray} \label{eq:eqnone}
 J(V)S\subseteq k_{\sigma}I_{\sigma}f(\sigma,\sigma^{-1}).
 \end{eqnarray}
 Since
 $k_{\sigma}\in U(S_N)$, if $f(\sigma,\sigma^{-1})\in N^2$ then, localizing both sides of \eqref{eq:eqnone} above at $N$ we get
 $J(S_N)\subseteq J(S_N)^2$, a contradiction, since $J(V)$ is a
 principal ideal of $V$. Therefore for each $\tau\in G$ and each
maximal ideal $M$ of $S$, $f(\tau,\tau^{-1})\not\in M^2$. Since the cocycle identity
$f^{\tau}(\tau^{-1},\tau\gamma)f(\tau,\gamma)=f(\tau,\tau^{-1})$ holds, we conclude that for all
$\tau,\gamma\in G$ and every maximal ideal $M$ of $S$, $f(\tau,\gamma)\not\in M^2$.

Conversely, suppose $f(\tau,\gamma)\not\in
M^2$ for all $\tau,\gamma\in G$, and every maximal ideal $M$ of $S$.
Let $O_l(J(A_f))=\{x\in\Sigma_f\mid xJ(A_f)\subseteq J(A_f)\}$. We
will first establish that $O_l(J(A_f))=A_f$, again emulating the
relevant steps in the proof of \cite[Theorem]{K4}. To achieve
this, it suffices to show that $O_l(J(A_f))=\sum_{\tau\in
G}Sk_{\tau}x_{\tau}$ for some $k_{\tau}\in K^{\#}$, and that
$I_{\tau}f(\tau,\tau^{-1})=J(V)S$ for each $\tau\in G$, where
$I_{\tau}$ is as in Lemma~\ref{lem:radical-lemma}. The second assertion follows from
Lemma~\ref{lem:radical-lemma}(3). As for the first one, we first note that $O_l(J(A_f))$
is a $V$-order in $\Sigma_f$, by \cite[Corollary 1.3]{K1}. By Lemma
2.5, $O_l(J(A_f))=\sum_{\tau\in G}Sk_{\tau}x_{\tau}$ for some
$k_{\tau}\in K^{\#}$ if and only if it is finitely generated over
$V$.

Since for all
$\tau,\gamma\in G$ and every maximal ideal $M$ of $S$ we have $f(\tau,\gamma)\not\in M^2$,
we conclude from Lemma~\ref{lem:overring-lemma} that
$f(G\times G)\subseteq U(R)$, hence $B_f$ is Azumaya over $W$.
Therefore $J(B_f)=J(W)B_f=J(W)(WA_f)=J(W)A_f\subseteq J(A_f)$, and
$A_f/J(B_f)$ is a $\tilde{V}$-order in $B_f/J(B_f)$. Since
$O_l(J(A_f))$ is a $V$-order containing $A_f$, $O_l(J(A_f))W$ is a
$W$-order containing $B_f$, so $O_l(J(A_f))W=B_f$, since $B_f$ is a
maximal $W$-order in $\Sigma_f$, and hence $O_l(J(A_f))\subseteq
B_f$. Therefore $O_l(J(A_f))/J(B_f)$ is a $\tilde{V}$-order in
$B_f/J(B_f)$, a central simple $\overline{W}$-algebra. Since
$\tilde{V}$ is a DVR of $\overline{W}$, $O_l(J(A_f))/J(B_f)$ must be
finitely generated over $\tilde{V}$, by \cite[Theorem 10.3]{R},
hence there exists $a_1,a_2,\ldots, a_n\in O_l(J(A_f))$ such that
$O_l(J(A_f))=a_1V+a_2V+\cdots+a_nV + J(B_f)=
a_1V+a_2V+\cdots+a_nV+A_f$, a finitely generated $V$-module. Thus
$O_l(J(A_f))=A_f$.

As in the proof of \cite[Lemma 4.11]{M}, we have
$$O_l(J(A_f/J(B_f)))=O_l(J(A_f)/J(B_f))=O_l(J(A_f))/J(B_f)=A_f/J(B_f),$$
where $O_l(J(A_f/J(B_f)))$ and $O_l(J(A_f)/J(B_f))$ are defined
accordingly.  Since $\tilde{V}$ is a DVR of $\overline{W}$,
$A_f/J(B_f)$ is a hereditary $\tilde{V}$-order in the central simple
$\overline{W}$-algebra $B_f/J(B_f)$, hence $A_f$ is semihereditary
by \cite[Lemma 4.11]{M}.
\end{proof}

We summarize these results as follows.

\begin{theorem} \label{thm:main-theorem}
Given a crossed-product order $A_f$,
\begin{enumerate}
    \item it is semihereditary if and only if
    for all
$\tau,\gamma\in G$ and every maximal ideal $M$ of $S$, $f(\tau,\gamma)\not\in M^2$; if and only if
    for each
$\gamma\in G$ and each maximal ideal $M$ of $S$, $f(\tau,\tau^{-1})\not\in M^2$.
    \item if $J(V)$ is a non-principal ideal of $V$, then $A_f$ is
    semihereditary if and only if it is Azumaya over $V$, if and only if $H=G$.
\end{enumerate}
\end{theorem}

We now lump together several corollaries of the theorem above,
generalizing results in \cite{K4}.

\begin{corollary} \label{cor:main-corollary}
\begin{enumerate}
\item Given a crossed-product order $A_f$,
    \begin{enumerate}\item it is a valuation ring if and only if given any maximal ideal $M$ of $S$,
     $f(\tau,\tau^{-1})\not\in M^2$ for each $\tau\in G$, and there exists
    a set of right coset representatives $g_1,g_2,\ldots,g_r$ of $D_M$
    in $G$ (i.e., $G$ is the disjoint union $\cup_iD_Mg_i$) such that
    for all $i$, $f(g_i,g_i^{-1})\not\in M$.
    \item if $V$ is indecomposed in $K$, then it is a valuation ring if and only if  for each $\tau\in G$,
    $f(\tau,\tau^{-1})\not\in J(S)^2$.
    \end{enumerate}
\item Suppose the crossed-product order $A_f$ is primary.
Then it is a valuation ring if and only if there exists a maximal
ideal $M$ of $S$ such that for each $\tau\in D_M$, $f(\tau,\tau^{-1})\not\in
M^2$.
\item Suppose the crossed-product order $A_f$ is
semihereditary. Then $A_{f_{L,U}}$ is a semihereditary order in
$\Sigma_{f_{L,U}}$ for each intermediate field $L$ of $F$ and $K$,
and every valuation ring $U$ of $L$ lying over $V$.
\item Suppose the crossed-product order $A_f$ is semihereditary.
Then $A_{f_M}$ is a valuation ring of $\Sigma_{f_M}$ for each maximal
ideal $M$ of $S$.
\end{enumerate}
\end{corollary}

We end by observing yet another peculiarity of these
crossed-product orders. The proposition below not only strengthens Lemma~\ref{lem:fgextremal-lemma}(2) when the $V$-order
$A$ is taken to be the crossed-product order $A_f$, but also generalizes \cite[Proposition
2.10]{HM} to the case where $V$ is not necessarily indecomposed in
$K$.

\begin{proposition} \label{prop:overring-prop}
Suppose the crossed-product order $A_f$ is
extremal and $W$ is a valuation ring of $F$ with $V\subsetneqq
W$. Then $WA_f$ is Azumaya over $W$.
\end{proposition}

\begin{proof} This follows from Lemma~\ref{lem:overring-lemma} and Theorem~\ref{thm:main-theorem}. \end{proof}

 \bibliographystyle{amsplain}

\noindent
Department of Mathematics\\ Faculty of Science\\
Universiti Brunei Darussalam\\ Bandar Seri
Begawan BE1410\\ BRUNEI.
\\

\end{document}